\documentclass{article}
\usepackage{amsmath,amsthm,amssymb,float,graphicx,geometry}
\usepackage{bbm,bbding,amssymb,pifont,mathrsfs,amsfonts,graphicx,subfigure,mathtools,color}%
\usepackage{amscd,array,enumerate,dsfont,texdraw,tikz,multicol,bbm,authblk}
\usepackage{algorithmic}
\newtheorem{algorithm}{Algorithm}
\usepackage{algorithm}
\usepackage{footnote}
\usepackage{lipsum}
\usepackage{amsmath}
\usepackage{cases}
\usepackage{fancyhdr}
\usepackage{CJK}
\usepackage{bm}
\usepackage{framed}
\usepackage{listings}
\usepackage{multirow}
\allowdisplaybreaks[4]
\usetikzlibrary{shapes.geometric, arrows}

\newtheorem{assumption}{Assumption}
\newtheorem{lem}{Lemma}
\newtheorem{thm}{Theorem}

\newtheorem{remark}{Remark}

\tikzstyle{startstop} = [rectangle, rounded corners, minimum width = 2cm, minimum height=1cm,text centered, draw = black, fill = red!40]
\tikzstyle{io} = [rectangle, rounded corners, minimum width = 2cm, minimum height=1cm,text centered, draw = black, fill = blue!40]
\tikzstyle{process} = [rectangle, rounded corners, minimum width = 2cm, minimum height=1cm,text centered, draw = black, fill = yellow!50]
\tikzstyle{decision} = [rectangle, rounded corners, minimum width = 2cm, minimum height=1cm,text centered, draw = black, fill = green!40]
\tikzstyle{arrow} = [->,>=stealth]
\title{Distributed Non-Bayesian Learning for Games with Incomplete Information}
\author[a,b]{Shijie Huang}
\author[c]{Jinlong Lei}
\author[c,a]{Yiguang Hong}
\affil[a]{Key Laboratory of Systems and Control, Academy of Mathematics and Systems Science, Chinese Academy of Sciences}
\affil[b]{School of Mathematical Sciences, University of Chinese Academy of Sciences}
\affil[c]{Department of Control Science and Engineering \& Shanghai Research Institute for Intelligent Autonomous Systems, Tongji University}

\date{}
\begin{document}
\maketitle
\definecolor{shadecolor}{rgb}{0.9,0.9,0.9}
\begin{abstract}
	We consider distributed learning problem in games with an unknown cost-relevant parameter, and aim to find the Nash equilibrium while learning the true parameter. Inspired by the social learning literature, we propose a distributed non-Bayesian rule to learn the parameter (each agent maintains a belief on the parameter and updates the belief according to the received noisy cost), combined with best response dynamics for strategy update. The difficulty of the analysis lies mainly in the fact that the parameter learning process and strategy update process are coupled. We first prove that agents' beliefs converge to a common belief and the strategy profiles converge to a Nash equilibrium under this common belief. On this basis, we further show that the beliefs eventually concentrate on the true parameter.
\end{abstract}



\section{Introduction}
Game model has attracted much research attention due to its wide applications in smart grid \cite{maharjan2013dependable}, wireless communication network \cite{han2012game}, traffic control \cite{groot2014toward} and machine learning \cite{goodfellow2020generative}. Equilibrium learning is a promising topic in game theory because it provides a rational way for agents to make decisions. Particularly, in some complex situations, agents usually need to make decisions when faced with an unknown environment. For example, the true value of a stock is unknown in a stock market \cite{eksin2012learning}; the delay of each route is unknown in a traffic routing problem \cite{meigs2017learning}; the location of target is unknown in a robotic target covering problem \cite{arslan2007autonomous}. Therefore, developing learning algorithms in such game models with incomplete information is crutial and has gained increasing popularity in recent years. 

The common framework of learning dynamics in games with an unknown cost-relevant environment (parameter) consists of two processes: parameter learning and strategy update. In the parameter learning process, agents update the estimates of the parameter according to received historical information. In the strategy update process, agents make decision given the current estimates of the parameter. Distinct from the learning dynamics for misspecified game model \cite{jiang2017distributed,lei2020asynchronous}, parameter learning and strategy update are coupled in this framework, which brings additional difficulty to the convergence analysis. Inspired by the social learning in networks \cite{gale2003bayesian,rosenberg2009informational}, the authors in \cite{eksin2012learning} studied the limit behaviour of rational learning process in linear games with an unknown parameter. Along another line of research, \cite{meigs2017learning} and \cite{meigs2019learning} designed learning dynamics for stochastic routing games with unknown latency function and stochastic network aggregative games with unknown network parameter, respectively, by utilizing the least-squares estimation technique. 

The aforementioned works assume that the unknown parameter influences the cost function in a special structure (such as a linear structure). Learning dynamics for game models with general parameter structure are rare in the literature. To the best of our knowledge, there are two works devoted to this kind of problem. On the premise that the parameter learning process converges at a rate faster than $O(\log t/t)$, \cite{eksin2017distributed} designed a distributed fictitious play dynamics for finite potential games with incomplete information. More recently, \cite{wu2020multi} proposed a Bayesian learning dynamics that does not depend on the structure of parameter affecting the cost. The learning process in \cite{wu2020multi} assumes the existence of an information system for centrally updating the belief about the parameter. However, such an information system may not exsit in some industries. This may be the case, for example, using a central system to collect information from all traders in the stock market may be impossible. A more rational way is that traders estimate the true value of the stock distributedly based on local information. This paper addresses distributed learning problem in games with incomplete information, where each agent updates the belief about the parameter according to private signals.   

In the absence of a central system, it is computationally expensive for each agent to deduce other agents' information in a fully Bayesian fashion. In contrast, non-Bayesian learning rules are usually more effective when dealing with learning problmes in large-scale networks because of its low computational burden \cite{jadbabaie2012non}. Following up on the seminal work of Ali Jadbabaie et al. \cite{jadbabaie2012non} on distributed learning, various non-Bayesian learning rules have been developed in social learning or distributed parameter estimation literature. This type of learning rules share a common framework, where each agent first performs a Bayesian update based on the private information and then incorporates the neighbors' beliefs in a simple way. For example, \cite{lalitha2018social} designed a non-Bayesian update rule for distributed hypothesis testing by averaging the neighbors' log-beliefs and analyzed the convergence rate. To reduce the communication burden, \cite{shahrampour2015switching} developed a switching learning rule between Bayesian and non-Bayesian regimes. Furthermore, \cite{parasnis2022non} studied the performance of non-Bayesian rule proposed by \cite{jadbabaie2012non} in random digraphs. \cite{ntemos2022social} and \cite{bordignon2021adaptive} proposed adaptive non-Bayesian rules to deal with social learning problems with disparate hypotheses and variable true hypotheses, respectively. In addition, \cite{nedic2017fast} considered the scenerio where the agents might have incosistent hypotheses from an optimization point of view and proposed an effective non-Bayesian learning rule for time-varying communication graphs.   

In this paper, we design learning dynamics for game models with incomplete information by leveraging the idea of non-Bayesian learning. In our learning process, each agent first performs a local \emph{tempered} Bayesian update to form a posterior belief according to received noisy cost. Then, agents exchange the posterior belief through a communication network and geometrically average the neighbors' beliefs to achieve belief consensus. Based on the current belief, agents select best response strategies for minimizing the expected cost function. Our contributions are as follows:
\begin{itemize}
	\item We propose a novel distributed non-Bayesian learning dynamics for NE seeking in games with incomplete information, which generalizes the learning rule for social learning considered in \cite{lalitha2018social}. Moreover, our dynamics does not require an information system to collect the noisy costs of all agents as in \cite{wu2020multi}.
	\item We prove that agents' beliefs about the parameter converge to a common belief and that the strategy profiles converge to a Nash equilibrium under this common belief. Compared with \cite{eksin2017distributed}, the beliefs do not necessarily converge to the common belief faster than $O(\log t/t)$ in our algorithm.
	\item We overcome the difficulty of coupling parameter learning and strategy update by combining the strong law of large numbers and Toeplitz' lemma, thus futher showing that the beliefs about the parameter converge to the true parameter.
\end{itemize} 

\noindent{\bf Notations}. Denote by $\mathbb{R}^n$ the $n$-dimensional real Euclidean space and $\textbf{1}_n$ the $n$-dimensional vector of all ones. For a matrix $A$, $A(i,j)$ denotes the element in the $i$th row and $j$th column. $\Delta(S)$ denotes the set of probability distributions on a set $S$. For a random variable $x$ with probability distribution $\mu$, denote by $\mathbb{E}_{\mu}[x]$ the expectation of $x$. Denote by $D_{KL}(\mu_1\|\mu_2)$ the Kullback-Leibler divergence between two probability distributions $\mu_1$ and $\mu_2$. An undirected graph is characterized by the $2$-tuple $\mathcal{G} = (\mathcal{V},\mathcal{E})$, where $\mathcal{V} = \{1,\dots,n\}$ is the set of nodes and $\mathcal{E}\subset\mathcal{V}\times\mathcal{V}$ is the set of edges. A path from $i_1$ to $i_p$ is an alternating sequence $i_1e_1\cdots i_{p-1}e_{p-1}i_p$ such that $e_r = (i_r,i_{r+1})\in\mathcal{E}$ $(r = 1,\dots,p-1)$. An undirected graph is connected if there is a path between any pair of distinct nodes.

\section{Problem Formulation}
Consider a group of $N$ agents that repeatedly play a stage game $G = (\mathcal{N},\{\mathcal{X}_i\},\{u_i\},\Theta)$ with incomplete information. In game $G$, the strategy variable of each player $i\in \mathcal{N}:=\{1,\dots,N\}$ is denoted by $x_i$ belonging to a convex and compact set $\mathcal{X}_i$. Define $x := \text{col}\{x_1,\dots,x_N\}\in\mathcal{X}:=\prod_{i=1}^N\mathcal{X}_i$ and $x_{-i}:=\text{col}\{x_1,\dots,x_{i-1},x_{i+1},\dots,x_N\}\in\prod_{j\neq i}\mathcal{X}_j$, respectively, as the strategy profile and the strategies of all agents except $i$. The important feature of this game is that the cost function of each agent $u_i:\mathcal{X}\times\Theta\to\mathbb{R}$ depends not only on a strategy profile, but also on an unknown parameter $\theta$ in a finite set $\Theta:=\{\theta_1,\dots,\theta_M\}$. Assume that the costs are realized with noises given a strategy profile $x$ and a parameter $\theta$. Specifically, let the realized cost of agent $i$ be
\[y_i = u_i(x,\theta) + \epsilon_i(x,\theta),\]
where $\epsilon_i(x,\theta)$ is the noise term with zero mean. Denote by $f_i(y_i|x,\theta)$ the likelihood function of $y_i$. We make the following assumption on the likelihood function.
\begin{assumption}\label{asm1}
	There exists a positive constant $L$ such that
	\[\max_i\max_{\theta',\theta''\in\Theta}\max_{x\in\mathcal{X}}\sup_{y_i}\left|\log\frac{f_i(y_i|x, \theta')}{f_i(y_i|x,\theta'')}\right|\le L.\]
	In addition, for each $i\in\mathcal{N}$, $f_i(y_i|x,\theta)$ is continuous in $x$ for all $\theta\in\Theta$.
\end{assumption}
Assumption \ref{asm1} requires that every realized cost has bounded information content, which is a standard assumption for the convergence of the beliefs in the social learning literature \cite{lalitha2018social, shahrampour2015distributed}. Also, the continuity assumption is crutial to ensure the convergence of the strategies in the game setting \cite{wu2020multi}.

Moreover, we denote the true parameter by $\theta^{\ast}\in\Theta$. Then, the associated Nash equilibrium is defined by a strategy profile $x^{\ast}$ such that no agent can benefit from deviating unilaterally, i.e., for $i\in\mathcal{N}$,
\begin{equation}\label{NE_def}
	u_i(x_i^{\ast},x_{-i}^{\ast},\theta^{\ast})\le u_i(x_i,x_{-i}^{\ast},\theta^{\ast}),\ \text{for all}\ x_i\in\mathcal{X}_i.
\end{equation}

To measure the explanatory quality of the parameters in the set $\Theta$, we consider the Kullback–Leibler (KL) divergence between the distribution of the realized costs conditioned over the strategies and parameters. Recall that the KL divergence between any two probability distributions $P$ and $P'$ is defined by
\begin{equation}\label{KL-div_def}
	D_{KL}(P\|P'):= \mathbb{E}_{P}\left[\log\frac{P}{P'}\right],
\end{equation}
which equals to $0$ if and only if $P = P'$ with probability $1$. Inspired by \cite{wu2020multi}, for each agent $i\in\mathcal{N}$ and a given strategy profile $x\in\mathcal{X}$, we define a set of cost-equivalent parameters
\[\bar{\Theta}_i(x):= \{\theta\in\Theta: D_{KL}(f_i(y_i|x,\theta^{\ast})\|f_i(y_i|x,\theta)) = 0\}.\]
In other words, given a strategy profile $x$, the parameters in the set $\bar{\Theta}_i(x)$ are locally indistinguishable to agent $i$. To learn the true parameter, we generalize the global identifiability assumption in the social learning literature \cite{lalitha2018social, shahrampour2013exponentially} to the game setting.
\begin{assumption}\label{asm2}
	For every $\theta\neq\theta^{\ast}$, there is at least one agent $i\in\mathcal{N}$ for which the KL divergence $D_{KL}(f_i(y_i|x,\theta^{\ast})\|f_i(y_i|x,\theta))$ is strictly positive for all $x\in\mathcal{X}$. 
\end{assumption}
Assumption \ref{asm2} guarantees that for all $x\in\mathcal{X}$,
\[\bar{\Theta}_1(x)\cap\bar{\Theta}_2(x)\cap\cdots\cap\bar{\Theta}_N(x) = \{\theta^{\ast}\}.\]
Therefore, although a single agent may not distinguish $\theta^{\ast}$ from all other parameters, the true parameter is globally identifiable from the standpoint of the global game model. 

The problem to be addressed in this paper is how agents learn the Nash equilibrium distributedly when the true parameter is unknown. Each agent forms its own belief on the parameter and can exchange information with other agents through a communication network. Agents attempts to learn the true parameter $\theta^{\ast}$ by combining their \emph{local knowledge} (realized costs) with the information received from the network. 

We model the communication network via an undirected graph $\mathcal{G}(\mathcal{N},\mathcal{E})$, where $\mathcal{N}$ is the node set and $\mathcal{E}$ is the edge set. Denote by $\mathcal{N}_i:=\{j\in\mathcal{N}: (j,i)\in\mathcal{E}\}$ the set of neighbors of agent $i$. Each agent can receive information from its neighbors. Moreover, we let $W = [w_{ij}]\in\mathbb{R}^{N\times N}$ denote the associated adjacency matrix, which defines the weights that each agent assigns to neighbors' information such that $w_{ij} > 0$ if and only if $j\in\mathcal{N}_i$. We require the following assumptions which have been previously used in distributed optimization \cite{shi2015extra,duchi2011dual} and game theory \cite{lei2018linearly,huang2022linearly}.
\begin{assumption}\label{asm3}
	The undirected communication graph $\mathcal{G}(\mathcal{N},\mathcal{E})$ is connected and $W$ is doubly stochastic, i.e., $\sum_{j=1}^Nw_{ij} = \sum_{i=1}^Nw_{ij} = 1$.
\end{assumption}
The following lemma \cite{rosenthal1995convergence} provides a mixing rate of the adjacency matrix and plays an important role in the convergence analysis.
\begin{lem}
	Denote by $\lambda_i(W)$ the $i$-largest eigenvalue of the mixing matrix. Under Assumption \ref{asm3}, $W$ satisfies
	\[\left|W^{t}(i,j) - \frac{1}{N}\right|\le \lambda_{\max}(W)^t,\]
	where $\lambda_{\max}(W)\triangleq\max\{|\lambda_N(W)|,|\lambda_2(W)|\}\in (0,1)$. 
\end{lem} 
\section{Algorithm Design}
Our learning process consists of two parts: the evolution of agents' beliefs about the unknown parameter and the update of agents' strategies. At each time step $t$, each agent $i$ maintains a private belief $\mu_i^{(t)}\in\Delta(\Theta)$ and takes action $x_i^{(t)}$. The cost $y_i^{(t)}$ realized according to $f_i(y_i^{(t)}|x^{(t)},\theta)$ when the parameter is $\theta$ and each agent forms a posterior belief $b_i^{(t)}$ based on the cost. Then, agents exchange the posterior beliefs with their neighbors and updates the private beliefs using a non-Bayesian rule.

Given the current private belief $\mu_i^{(t+1)}$, agent $i$ evaluates its expected cost by
\begin{equation}\label{expected_payoff}
	u_i(x,\mu_i^{(t+1)}):= \sum_{\theta\in\Theta}\mu_i^{(t+1)}(\theta)u_i(x,\theta).
\end{equation}
Further, assuming that the rivals' strategies are fixed as $x_{-i}^{(t)}$, the best-response mapping of agent $i$ is defined by
\begin{equation}
	\text{BR}_i(x_{-i}^{(t)},\mu_i^{(t+1)}):=\arg\min_{x_i\in \mathcal{X}_i}u_i(x_i,x_{-i}^{(t)},\mu_i^{(t+1)}).
\end{equation}
We thoroughly describe our learning process in Algorithm \ref{alg1}.
\begin{algorithm}[H]
	\caption{Distributed Non-Bayesian Learning in Games}
	\label{alg1}
	\textbf{Input}: Non-increasing nonnegative step-size sequence $\{\alpha^{(t)}\ge 0\}_{t\ge 0}$, weight matrix $W$\\
	\textbf{Initialize}: $x_{i}^{(0)}  \in \mathcal{X}_i$ for each $i \in\mathcal{N}$, and $\mu_i^{(0)} = \frac{1}{M}\textbf{1}_M$.\\
	\qquad{\bf Belief Update:} For each agent $i\in\mathcal{N}$, and $k = 1,\dots,M$,\\
	\indent $\qquad$ Update local posterior belief:
	\begin{equation}\label{Bayesian_update}
		b_i^{(t)}(\theta_k) = \frac{f_i(y_i^{(t)}|x^{(t)},\theta_k)^{\alpha^{(t)}}\mu_i^{(t)}(\theta_k)}{\sum_{\theta\in\Theta}f_i(y_i^{(t)}|x^{(t)},\theta)^{\alpha^{(t)}}\mu_i^{(t)}(\theta)}
	\end{equation} 	 
	\indent $\qquad$ Receive information $b_j^{(t)}(\theta_k)$ from $j\in\mathcal{N}(i)$ and perform a non-Bayesian rule to update the private belief:
	\begin{equation}\label{non-Bayesian_update}
		\mu_i^{(t+1)}(\theta_k) = \frac{\exp(\sum_{j=1}^Nw_{ij}\log b_j^{(t)}(\theta_k))}{\sum_{\theta\in\Theta}\exp(\sum_{j=1}^Nw_{ij}\log b_j^{(t)}(\theta))}
	\end{equation}
	{\bf Strategy Update:} For each agent $i\in\mathcal{N}$,\\
	\indent $\qquad$ Observe the opponents' strategies $x_{-i}^{(t)}$ and perform a best-response update:
	\begin{equation}\label{strategy_update}
		x_i^{(t+1)} \in (1 - \alpha^{(t)})x_i^{(t)} + \alpha^{(t)}\text{BR}_i(x_{-i}^{(t)},\mu_i^{(t+1)})
	\end{equation}	 
\end{algorithm}
Note that the update rule \eqref{Bayesian_update} of posterior belief is different from the traditional Bayes rule. Such a generalized posterior is called tempered posterior distribution, which is easier to theoretical analysis \cite{bhattacharya2019bayesian} and more robust to model misspecification \cite{grunwald2012safe}. Tempered posterior has also been used in social learning \cite{paritosh2022distributed} and stochastic bandit problem \cite{lalitha2021bayesian}.
\begin{remark}
	Compared to the non-Bayesian learning algorithm for social learning \cite{lalitha2018social}, in a game setting, the likelihood functions of the realized costs $\{y_i^{(t)}\}_{t\ge 0}$ depends not only on the unknown parameter, but also on the strategy profile $x^{(t)}$. And hence, the costs are not independent and identically distributed with respect to time $t$. In addition, different from the learning algorithm in \cite{wu2020multi} that requires an information system to centrally perform Bayesian update, agents employ a non-Bayesian rule to distributedly learn the unknown parameter in our algorithm.   
\end{remark}
\section{Main Results}
In this section, we present the convergence of agents' beliefs to the true parameter and the convergence of the strategy profile to the Nash equilibrium.
\subsection{Belief Convergence}
In this part, we show that agents' beliefs reach consensus to a common belief $\mu$. Before presenting the proof, we make the following assumption on the step-size sequence. 
\begin{assumption}\label{asm4}
	The step-size sequence $\{\alpha^{(t)}\}$ satisfies $0<\alpha^{(t)}<1$, $\sum_{t=0}^{\infty}\alpha^{(t)} = \infty$, and $\sum_{t=0}^{\infty}(\alpha^{(t)})^2 < \infty$.
\end{assumption}
First, we prove the convergence of the average belief ratio $\frac{1}{N}\sum_{i=1}^N\frac{\mu_i^{(t)}(\theta)}{\mu_i^{(t)}(\theta^{\ast})}$. Let $\mathcal{F}_t:=\sigma\{\mu_i^{(0)},x_i^{(s)},y_i^{(s)},0\le s\le t-1, i\in\mathcal{N}\}$ denote the $\sigma$-field containing the past information about all agents. 
\begin{lem}
	Let Assumption \ref{asm4} hold. Then, the sequence $\frac{1}{N}\sum_{i=1}^N\frac{\mu_i^{(t)}(\theta_k)}{\mu_i^{(t)}(\theta^{\ast})}$ converges with probability $1$ to some non-negative random variable $\nu_k$ for all $\theta_k\in\Theta$.
\end{lem}
\begin{proof}
	By the belief update rules \eqref{Bayesian_update} and \eqref{non-Bayesian_update},
	\begin{align*}
		\frac{\mu_i^{(t+1)}(\theta_k)}{\mu_i^{(t+1)}(\theta^{\ast})} &=\exp\left( \sum_{j=1}^Nw_{ij}\log\frac{b_j^{(t)}(\theta_k)}{b_j^{(t)}(\theta^{\ast})}\right)\\
		&\le \sum_{j=1}^Nw_{ij}\frac{f_j(y_j^{(t)}|x^{(t)},\theta_k)^{\alpha^{(t)}}\mu_j^{{(t)}}(\theta_k)}{f_j(y_j^{(t)}|x^{(t)},\theta^{\ast})^{\alpha^{(t)}}\mu_j^{{(t)}}(\theta^{\ast})},
	\end{align*}
	where the inequality is followed by $\sum_{j=1}^Nw_{ij} = 1$ and the Jensen's inequality. Furthermore, using $\sum_{i=1}^Nw_{ij} = 1$, we get
	\[\sum_{i=1}^N\frac{\mu_i^{(t+1)}(\theta_k)}{\mu_i^{(t+1)}(\theta^{\ast})}\le \sum_{i=1}^N\frac{f_j(y_i^{(t)}|x^{(t)},\theta_k)^{\alpha^{(t)}}\mu_i^{(t)}(\theta_k)}{f_i(y_i^{(t)}|x^{(t)},\theta^{\ast})^{\alpha^{(t)}}\mu_i^{(t)}(\theta^{\ast})}.\]
	Thus, by taking conditional expectation and noting that $\mu_i^{(t)}(\theta)$ is $\mathcal{F}_t$-measurable for all $\theta\in\Theta$, we derive
	\[\mathbb{E}\left[\sum_{i=1}^N\frac{\mu_i^{(t+1)}(\theta_k)}{\mu_i^{(t+1)}(\theta^{\ast})}|\mathcal{F}_t\right] \le \sum_{i=1}^N\frac{\mu_i^{(t)}(\theta_k)}{\mu_i^{(t)}(\theta^{\ast})}\mathbb{E}\left[\left(\frac{f_i(y_i^{(t)}|x^{(t)},\theta_k)}{f_i(y_i^{(t)}|x^{(t)},\theta^{\ast})}\right)^{\alpha^{(t)}}|\mathcal{F}_t\right],\]
	where the expectation is taken on the probability distribution $\prod_{j=1}^Nf_j(y_j|x,\theta^{\ast})$. Since $x^{\alpha}$ is a concave function when $0<\alpha<1$, Jensen's inequality implies that
	\begin{align*}
		\mathbb{E}\left[\sum_{i=1}^N\frac{\mu_i^{(t+1)}(\theta_k)}{\mu_i^{(t+1)}(\theta^{\ast})}|\mathcal{F}_t\right] &\le \sum_{i=1}^N\frac{\mu_i^{(t)}(\theta_k)}{\mu_i^{(t)}(\theta^{\ast})}\mathbb{E}\left[\frac{f_i(y_i^{(t)}|x^{(t)},\theta_k)}{f_i(y_i^{(t)}|x^{(t)},\theta^{\ast})}|\mathcal{F}_t\right]^{\alpha^{(t)}}\\
		&=  \sum_{i=1}^N\frac{\mu_i^{(t)}(\theta_k)}{\mu_i^{(t)}(\theta^{\ast})}\left(\int_{y_i^{(t)}}f_i(y_i^{(t)}|x^{(t)},\theta^{\ast})\frac{f_i(y_i^{(t)}|x^{(t)},\theta_k)}{f_i(y_i^{(t)}|x^{(t)},\theta^{\ast})}dy_i^{(t)}\right)^{\alpha^{(t)}}\\
		&= \sum_{i=1}^N\frac{\mu_i^{(t)}(\theta_k)}{\mu_i^{(t)}(\theta^{\ast})},
	\end{align*}
	where the first equality follows from that conditioned on the current action profile $x^{(t)}$, $\frac{f_i(y_i^{(t)}|x^{(t)},\theta_k)}{f_i(y_i^{(t)}|x^{(t)},\theta^{\ast})}$ is independent of $\mathcal{F}_t$. Therefore, $\sum_{i=1}^N\frac{\mu_i^{(t)}(\theta_k)}{\mu_i^{(t)}(\theta^{\ast})}$ is a positive supermartingale. By the supermartingale convergence theorem \cite[Theorem 9.4.4]{chung2001course}, we conclude that $\frac{1}{N}\sum_{i=1}^N\frac{\mu_i^{(t)}(\theta)}{\mu_i^{(t)}(\theta^{\ast})}$ converges with probability $1$.
\end{proof}
Next, we establish the consensus of the log-belief ratio $\log\frac{\mu_i^{(t)}(\theta)}{\mu_i^{(t)}(\theta^{\ast})}$.
\begin{lem}
	Let Assumptions \ref{asm1},\ref{asm3},\ref{asm4} hold. Then the agents' log-belief ratios reach consensus, i.e., for all $\theta_k\in\Theta$,
	\begin{equation}
		\left|\log\frac{\mu_i^{(t)}(\theta_k)}{\mu_i^{(t)}(\theta^{\ast})} - \frac{1}{N}\sum_{i=1}^N\log\frac{\mu_i^{(t)}(\theta_k)}{\mu_i^{(t)}(\theta^{\ast})}\right|\to 0.
	\end{equation} 
\end{lem}
\begin{proof}
	Again, by the belief update rules \eqref{Bayesian_update} and \eqref{non-Bayesian_update}, we derive
	\begin{align}
		\log\frac{\mu_i^{(t+1)}(\theta_k)}{\mu_i^{(t+1)}(\theta^{\ast})} &= \sum_{j=1}^Nw_{ij}\log\frac{b_j^{(t)}(\theta_k)}{b_j^{(t)}(\theta^{\ast})}\notag\\
		&= \sum_{j=1}^Nw_{ij}\log\frac{\mu_j^{(t)}(\theta_k)}{\mu_j^{(t)}(\theta^{\ast})} + \alpha^{(t)}\sum_{j=1}^Nw_{ij}\log\frac{f_j(y_j^{(t)}|x^{(t)},\theta_k)}{f_j(y_j^{(t)}|x^{(t)},\theta^{\ast})}\notag\\
		&= \sum_{j=1}^N\sum_{\tau=1}^tW^{\tau}(i,j)\alpha^{(t-\tau +1)}\log \frac{f_j(y_j^{(t-\tau+1)}|x^{(t-\tau +1)},\theta_k)}{f_j(y_j^{(t-\tau+1)}|x^{(t-\tau +1)},\theta^{\ast})} + \sum_{j=1}^NW^{t+1}(i,j)\log\frac{\mu_j^{(0)}(\theta_k)}{\mu_j^{(0)}(\theta^{\ast})}\notag\\
		&= \sum_{j=1}^N\sum_{\tau=1}^tW^{\tau}(i,j)\alpha^{(t-\tau +1)}\log \frac{f_j(y_j^{(t-\tau+1)}|x^{(t-\tau +1)},\theta_k)}{f_j(y_j^{(t-\tau+1)}|x^{(t-\tau +1)},\theta^{\ast})},\label{log_belief_ratio_1}
	\end{align}
	where the last equality follows from $\mu_i^{(0)} = \frac{1}{M}\textbf{1}_M$. By the double stochasticity of $W$, 
	\begin{equation}\label{average_log_belief_ratio}
		\frac{1}{N}\sum_{i=1}^N\log\frac{\mu_i^{(t+1)}(\theta_k)}{\mu_i^{(t+1)}(\theta^{\ast})} = \frac{1}{N}\sum_{i=1}^N\sum_{\tau=1}^t\alpha^{(t-\tau +1)}\log \frac{f_i(y_i^{(t-\tau+1)}|x^{(t-\tau +1)},\theta_k)}{f_i(y_i^{(t-\tau+1)}|x^{(t-\tau +1)},\theta^{\ast})},
	\end{equation}
	As a result,
	\begin{align*}
		&\quad\left|\log\frac{\mu_i^{(t+1)}(\theta_k)}{\mu_i^{(t+1)}(\theta^{\ast})} - \frac{1}{N}\sum_{i=1}^N\log\frac{\mu_i^{(t+1)}(\theta_k)}{\mu_i^{(t+1)}(\theta^{\ast})}\right|\\
		&\le \sum_{j=1}^N\sum_{\tau=1}^t\alpha^{(t-\tau+1)}\left|W^{\tau}(i,j) - \frac{1}{N}\right|\left|\log \frac{f_j(y_j^{(t-\tau+1)}|x^{(t-\tau+1)},\theta_k)}{f_j(y_j^{(t-\tau+1)}|x^{(t-\tau+1)},\theta^{\ast})}\right|.
	\end{align*}
	By the connectivity of the communication graph, $\left|W^{\tau}(i,j) - \frac{1}{N}\right|\le\lambda_{\max}(W)^{\tau}$. Thus, by Assumption \ref{asm1},
	\[\left|\log\frac{\mu_i^{(t+1)}(\theta_k)}{\mu_i^{(t+1)}(\theta^{\ast})} - \frac{1}{N}\sum_{i=1}^N\log\frac{\mu_i^{(t+1)}(\theta_k)}{\mu_i^{(t+1)}(\theta^{\ast})}\right|\le NL\sum_{\tau = 1}^t\alpha^{(t-\tau+1)}\lambda_{\max}(W)^{\tau},\]
	which converges to $0$ by Lemma 7\footnote{Let $0<\beta<1$ and let $\{\gamma_k\}$ be a positive scalar sequence. Assume that $\lim_{k\to\infty}\gamma_k = 0$. Then $\lim_{k\to\infty}\sum_{l=0}^k\beta^{k-l}\gamma_l = 0$.} of \cite{nedic2010constrained}. 
\end{proof}
Lemma 3 provides the reason why we consider a tempered posterior distribution in Algorithm \ref{alg1}. Based on Lemma 2 and Lemma 3, we now present our first main result in the following theorem.
\begin{thm}\label{thm1}
	Let Assumptions \ref{asm1}, \ref{asm3}, \ref{asm4} hold. The belief sequence $\{\mu_i^{(t)}\}_{t\ge 0}$ of each agent generated by Algorithm \ref{alg1} converges almost surely to a common belief $\mu$ with the form
	\[\mu(\theta_k) = \frac{\nu_k}{\sum_{k=1}^M\nu_k},\ \text{for}\ k = 1,\dots,M\]
	where $\nu_k$ is defined in Lemma 2.
\end{thm}
\begin{proof}
	Performing an exponential operation on Lemma 3 yields
	\[\frac{\mu_i^{(t+1)}(\theta_k)}{\mu_i^{(t+1)}(\theta^{\ast})}\cdot\frac{1}{\exp\left(\frac{1}{N}\sum_{i=1}^N\log\frac{\mu_i^{(t+1)}(\theta_k)}{\mu_i^{(t+1)}(\theta^{\ast})}\right)}\to 1.\]
	Taking the average, we get
	\[\frac{1}{N}\sum_{i=1}^N\frac{\mu_i^{(t+1)}(\theta_k)}{\mu_i^{(t+1)}(\theta^{\ast})}\cdot\frac{1}{\exp\left(\frac{1}{N}\sum_{i=1}^N\log\frac{\mu_i^{(t+1)}(\theta_k)}{\mu_i^{(t+1)}(\theta^{\ast})}\right)}\to 1.\]
	Furthermore, taking the logarithm of both sides, we obtain
	\[\log \left(\frac{1}{N}\sum_{i=1}^N\frac{\mu_i^{(t+1)}(\theta_k)}{\mu_i^{(t+1)}(\theta^{\ast})}\right) - \frac{1}{N}\sum_{i=1}^N\log\frac{\mu_i^{(t+1)}(\theta_k)}{\mu_i^{(t+1)}(\theta^{\ast})} \to 0.\]
	Thus, by Lemma 2, 
	\begin{equation*}
		\frac{1}{N}\sum_{i=1}^N\log\frac{\mu_i^{(t+1)}(\theta_k)}{\mu_i^{(t+1)}(\theta^{\ast})}\to \log\nu_k,\ a.s.
	\end{equation*}
	Using Lemma 3 again, we get
	\begin{equation}\label{log-belief_ratio}
		\log\frac{\mu_i^{(t+1)}(\theta_k)}{\mu_i^{(t+1)}(\theta^{\ast})}\to \log\nu_k,\ a.s.
	\end{equation}
	On the other hand, by the belief update rules,
	\begin{align}
		\mu_i^{(t+1)}(\theta^{\ast}) &\overset{\eqref{non-Bayesian_update}}{=} \frac{\exp(\sum_{j=1}^Nw_{ij}\log b_j^{(t)}(\theta^{\ast}))}{\sum_{\theta\in\Theta}\exp(\sum_{j=1}^Nw_{ij}\log b_j^{(t)}(\theta))}\notag\\
		&= \left(1 + \sum_{\theta\neq\theta^{\ast}}\exp\left(\sum_{j=1}^Nw_{ij}\log\frac{b_j^{(t)}(\theta)}{b_j^{(t)}(\theta^{\ast})}\right)\right)^{-1}\notag\\
		&\overset{\eqref{Bayesian_update}}{=} \left(1 + \sum_{\theta\neq\theta^{\ast}}\exp\left(\sum_{j=1}^Nw_{ij}\alpha^{(t)}\log\frac{f_j(y_j^{(t)}|x^{(t)},\theta)}{f_j(y_j^{(t)}|x^{(t)},\theta^{\ast})} + \sum_{j=1}^Nw_{ij}\log\frac{\mu_j^{(t)}(\theta)}{\mu_j^{(t)}(\theta^{\ast})}\right)\right)^{-1}\label{belief_true_parameter}
	\end{align}
	Without loss of generality, we let $\theta^{\ast} = \theta_1$. By Assumption \ref{asm1} and Assumption \ref{asm4}, $\log\frac{f_j(y_j^{(t)}|x^{(t)},\theta)}{f_j(y_j^{(t)}|x^{(t)},\theta^{\ast})}$ is bounded and $\alpha^{(t)}\to 0$. Thus, for all $\theta\in\Theta$,
	\begin{equation}\label{bounded_term}
		\exp\left(\sum_{j=1}^Nw_{ij}\alpha^{(t)}\log\frac{f_j(y_j^{(t)}|x^{(t)},\theta)}{f_j(y_j^{(t)}|x^{(t)},\theta^{\ast})}\right)\to 1.
	\end{equation}
	And by \eqref{log-belief_ratio},
	\begin{equation}\label{consensus_term}
		\exp\left(\sum_{j=1}^Nw_{ij}\log\frac{\mu_j^{(t)}(\theta_k)}{\mu_j^{(t)}(\theta^{\ast})}\right)\to \nu_k,\quad a.s.
	\end{equation}
	Substituting \eqref{bounded_term} and \eqref{consensus_term} into \eqref{belief_true_parameter}, we obtain 
	\[\mu_i^{t+1}(\theta^{\ast})\to (1 + \sum_{k=2}^M\nu_k)^{-1},\ a.s.\]
	Further, applying \eqref{log-belief_ratio} yields
	\[\mu_i^{t+1}(\theta_k)\to \frac{\nu_k}{1 + \sum_{k=2}^M\nu_k},\ a.s.\]
	The assertion follows by noting that $\nu_1 = 1$ when $\theta^{\ast} = \theta_1$.
\end{proof}
Theorem \ref{thm1} shows that agents' beliefs converge to a common belief, which is not necessarily the true belief. Distinct from the centralized Bayesian update in \cite{wu2020multi} whose belief convergence is directly obtained through the martingale convergence theorem, we also need to prove that the beliefs of different agents reach consensus.
\subsection{Strategy Convergence}
Similar to \eqref{NE_def}, we define the Nash equilibrium of the game with a common belief $\mu$ by a strategy profile $x^{\ast}(\mu)$ satisfying
\[u_i(x_i^{\ast}(\mu),x_{-i}^{\ast}(\mu),\mu)\le u_i(x_i,x_{-i}^{\ast}(\mu),\mu),\ \text{for all}\ x_i\in\mathcal{X}_i\]
where the expected cost is defined by \eqref{expected_payoff}. Moreover, we consider an auxiliary trajectory $\bar{x}(t)$ generated by the following continuous-time best-response dynamics under the common belief $\mu$ \begin{equation}\label{strategy_common}
	\frac{d\bar{x}_i(t)}{dt} \in \text{BR}_i(\bar{x}_{-i}(t),\mu) - \bar{x}_i(t),\quad \bar{x}_i(0) = x_i^{(0)}.
\end{equation}
To study the strategy convergence of Algorithm \ref{asm1}, we require the following assumption on the cost function.
\begin{assumption}\label{asm5}
	The cost function $u_i(x_i,x_{-i},\theta)$ is continuous in $x$ and strictly convex in $x_i$ for any $\theta\in\Theta$.
\end{assumption}
Assumption \ref{asm5} guarantees that there exists a unique Nash equilibrium for the game with a common belief \cite{scutari2010convex} and for any $\mu'\in\Delta(\Theta)$ and any $x_{-i}\in\mathcal{X}_{-i}$, the best response mapping $\text{BR}_i(x_{-i},\mu')$ is single-valued. Further, from Berge's maximum theorem \cite{ok2011real}, $\text{BR}_i(x_{-i},\mu')$ is continuous in both $x_{-i}$ and $\mu'$. 
In addition, inspired by \cite{wu2020multi}, we impose the following assumption on the best-response dynamics \eqref{strategy_common}, which holds for potential games and dominance solvable games \cite{wu2020multi}.
\begin{assumption}\label{asm6}
	For any $x^{(0)}\in\mathcal{X}$, the solution of \eqref{strategy_common} converges to the Nash equilibrium of the game with the common belief $\mu$, i.e., $\lim_{t\to\infty}\bar{x}(t) = x^{\ast}(\mu)$. 
\end{assumption}
Our next main result is based on the following stochstic approximation conclusion \cite{borkar2009}.
\begin{lem}
	Consider a stochstic approximation scheme
	\[z_{n+1} = z_n + \gamma_n(h(z_n) + \xi_n),\]
	where $\{\xi_n\}$ is a bounded random sequence with $\xi_n\to 0$ almost surely. Assume that $h$ is Lipschitz continuous and the step-size sequence $\gamma_n$ satisfies $\sum_n\gamma_n = \infty$ and $\sum_n\gamma_n^2<\infty$. If $\sup_n\|z_n\|<\infty$ almost surely, then the sequence $\{z_n\}$ converges almost surely to the set of asymptotically stable equilibria of the dynamics $dz(t)/dt = h(z(t))$. 
\end{lem}
We proceed to derive the convergence of the strategy profiles $\{x^{(t)}\}_{t=1}^{\infty}$.
\begin{thm}
	Let Assumptions \ref{asm1}, \ref{asm3}-\ref{asm6} hold. Then, the strategy profile $x^{(t)}$ generated by Algorithm \ref{alg1}  converges almost surely to $x^{\ast}(\mu)$.
\end{thm}
\begin{proof}
	By \eqref{strategy_update}, we get
	\begin{align*}
		x_i^{(t+1)} &= (1 - \alpha^{(t)})x_i^{(t)} + \alpha^{(t)}\text{BR}_i(x_{-i}^{(t)},\mu_i^{(t+1)})\\
		&=x_i^{(t)} + \alpha^{(t)}(\text{BR}_i(x_{-i}^{(t)},\mu_i^{(t+1)}) - x_i^{(t)})\\
		&= x_i^{(t)} + \alpha^{(t)}\left(\text{BR}_i(x_{-i}^{(t)},\mu) - x_i^{(t)} + \text{BR}_i(x_{-i}^{(t)},\mu_i^{(t+1)}) - \text{BR}_i(x_{-i}^{(t)},\mu)\right)
	\end{align*}
	Note by Theorem 1 that $\mu_i^{(t+1)}\to \mu$ almost surely. Thus, using the continuity of $\text{BR}_i$, we obtain
	\[\text{BR}_i(x_{-i}^{(t)},\mu_i^{(t+1)}) - \text{BR}_i(x_{-i}^{(t)},\mu)\to 0, a.s.\]
	The conlusion is followed by applying Assumption \ref{asm4} and Assumption \ref{asm6} to Lemma 4.
\end{proof}
The strategy update rule \eqref{strategy_update} is similar to the best-response strategy considered in \cite{wu2020multi}, except that each agent uses its own belief in our algorithm. In fact, the equilibrium strategy and the corresponding convergence result in \cite{wu2020multi} are also applicable since we have obtained $\mu_i^{(t)}\to\mu,\ a.s.$.

\subsection{Convergence to the True Parameter}
In this part, we further prove that agents' beliefs converge to the true belief by combining Theorem 1 and Theorem 2. In general, the strong law of large numbers and McDiarmid’s inequality are employed in the social learning literature \cite{lalitha2018social,shahrampour2015distributed} to show the convergence of agents' beliefs to the true parameter. However, all these results rely on an assumption that the private signals (realized costs in our setting) are independent and identically distributed across the time $t$. And hence, these techniques cannot be directly applied to the game setting as the realized costs are not i.i.d due to the influence of the strategies. We instead use the following Toeplitz's lemma \cite{knopp1990theory} to develop a similar convergence result.
\begin{lem}
	Let $\{A_{nk},1\le k\le k_n\}_{n\ge 1}$ be a double array of positive numbers such that for fixed $k$, $A_{nk}\to 0$ when $n\to\infty$. Let $\{Y_n\}_{n\ge 1}$ be a sequence of real numbers. If $Y_n\to y$ and $\sum_{k=1}^{k_n}A_{nk}\to 1$ when $n\to \infty$, then $\lim_{n\to\infty}\sum_{k=1}^{k_n}A_{nk}Y_k = y$. 
\end{lem} 
Now we state our final main result in the following theorem.
\begin{thm}
	Suppose that Assumptions \ref{asm1}-\ref{asm6} hold. Let $\{\mu_i^{(t)}\}_{t\ge 0}$ be the belief sequence generated by Algorithm \ref{alg1}. Then, for each agent $i\in\mathcal{N}$, it holds that 
	\[\lim_{T\to\infty}\frac{1}{\sum_{t=1}^T\alpha^{(t)}}\log\frac{\mu_i^{(T+1)}(\theta^{\ast})}{\mu_i^{(T+1)}(\theta_k)}
	= Z(\theta^{\ast},\theta_k)\quad a.s. 
	\]
	where $Z(\theta^{\ast},\theta_k)=\frac{1}{N}\sum_{j=1}^ND_{KL}(f_j(y_j|x^{\ast}(\mu),\theta^{\ast})\|f_j(y_j|x^{\ast}(\mu),\theta_k))$ is the network divergence and $x^{\ast}(\mu)$ is defined in Assumption \ref{asm6}. In particular, we have $\mu_i^{(t)}(\theta^{\ast})\to 1$ almost surely.
\end{thm}
\begin{proof}
	Similar to \eqref{log_belief_ratio_1}, we get
	\begin{align*}
		&\quad\lim_{T\to\infty}\frac{1}{\sum_{t=1}^T\alpha^{(t)}}\log\frac{\mu_i^{(T+1)}(\theta^{\ast})}{\mu_i^{(T+1)}(\theta_k)}\\
		&= \lim_{T\to\infty}\frac{1}{\sum_{t=1}^T\alpha^{(t)}}\sum_{j=1}^N\sum_{t=1}^TW^{t}(i,j)\alpha^{(T-t+1)}\log\frac{f_j(y_j^{(T-t+1)}|x^{(T-t+1)},\theta^{\ast})}{f_j(y_j^{T-t+1}|x^{(T-t+1)},\theta_k)}\\
		&= \lim_{T\to\infty}\frac{1}{\sum_{t=1}^T\alpha^{(t)}}\sum_{j=1}^N\sum_{t=1}^T\alpha^{(T-t+1)}(W^{t}(i,j) - \frac{1}{N})\log\frac{f_j(y_j^{(T-t+1)}|x^{(T-t+1)},\theta^{\ast})}{f_j(y_j^{(T-t+1)}|x^{(T-t +1)},\theta_k)}\\
		&\quad + \lim_{T\to\infty}\frac{1}{\sum_{t=1}^T\alpha^{(t)}}\frac{1}{N}\sum_{j=1}^N\sum_{t=1}^T\alpha^{(T-t+1)}\log\frac{f_j(y_j^{(T-t+1)}|x^{(T-t+1)},\theta^{\ast})}{f_j(y_j^{(T-t+1)}|x^{(T-t +1)},\theta_k)}.
	\end{align*}
	Note by \eqref{log_belief_ratio_1} and \eqref{average_log_belief_ratio} that
	\begin{align*}
		&\quad\lim_{T\to\infty}\frac{1}{\sum_{t=1}^T\alpha^{(t)}}\sum_{j=1}^N\sum_{t=1}^T\alpha^{(T-t+1)}(W^{t}(i,j) - \frac{1}{N})\log\frac{f_j(y_j^{(T-t+1)}|x^{(T-t+1)},\theta^{\ast})}{f_j(y_j^{(T-t+1)}|x^{(T-t +1)},\theta_k)}\\
		&= \lim_{T\to\infty}\frac{1}{\sum_{t=1}^T\alpha^{(t)}}\left(\log\frac{\mu_i^{(T+1)}(\theta^{\ast})}{\mu_i^{(T+1)}(\theta_k)} - \frac{1}{N}\sum_{i=1}^N\log\frac{\mu_i^{(T+1)}(\theta^{\ast})}{\mu_i^{(T+1)}(\theta_k)}\right)\\
		&= 0,
	\end{align*}
	where the last equality follows from Lemma 3 and Assumption \ref{asm4}. As a result, combining the above two relations and denoting by $z_j^{(t)}(\theta^{\ast},\theta_k) = \log\frac{f_j(y_j^{(t)}|x^{(t)},\theta^{\ast})}{f_j(y_j^{(t)}|x^{(t)},\theta_k)}$ and $s^{(T)} = \sum_{t=1}^T\alpha^{(t)}$, we further derive
	\begin{align}
		&\quad\ \lim_{T\to\infty}\frac{1}{\sum_{t=1}^T\alpha^{(t)}}\log\frac{\mu_i^{(T+1)}(\theta^{\ast})}{\mu_i^{(T+1)}(\theta_k)}\notag\\
		&= \lim_{T\to\infty}\frac{1}{\sum_{t=1}^T\alpha^{(t)}}\frac{1}{N}\sum_{j=1}^N\sum_{t=1}^T\alpha^{(T-t+1)}\log\frac{f_j(y_j^{(T-t+1)}|x^{(T-t+1)},\theta^{\ast})}{f_j(y_j^{(T-t+1)}|x^{(T-t +1)},\theta_k)}\notag\\
		&= \frac{1}{N}\sum_{j=1}^N\lim_{T\to\infty}\frac{1}{s^{(T)}}\sum_{t=1}^T\alpha^{{(t)}}z_j^{(t)}(\theta^{\ast},\theta_k)\label{weighted_sum}
	\end{align}
	To consider the convergence of the weighted average of the random variables $z_j^{(t)}(\theta^{\ast},\theta_k)$, we first study the convergence of $\frac{1}{T}\sum_{t=1}^Tz_j^{(t)}(\theta^{\ast},\theta_k)$. By Theorem 2, $x^{(t)}\to x^{\ast}(\mu)$ almost surely. We define the following cumulative distribution functions
	\[F_j^{(t)}(z)\triangleq P\left\{z_j^{(t)}(\theta^{\ast},\theta_k)\le z\right\},\quad F_j^{\ast}(z)\triangleq P\left\{\log\frac{f_j(y_j|x^{\ast}(\mu),\theta^{\ast})}{f_j(y_j|x^{\ast}(\mu),\theta_k)}\le z\right\}.\]
	We establish the relationship between $F_j^{(t)}(z)$ and $F_j^{\ast}(z)$ in two steps.\\
	{\bf Step 1}: From $F_j^{(t)}(z)$ to uniform distribution on $[0,1]$\\
	\indent Define
	\[\Delta_j^{(t)}\triangleq F_j^{(t)}\left(z_j^{(t)}(\theta^{\ast},\theta_k)\right).\]
	\indent Then, $\Delta_j^{(t)}\in[0,1]$ and for any $\delta\in[0,1]$,
	\[P\{\Delta_j^{(t)}\le\delta\} = P\left\{z_j^{(t)}(\theta^{\ast},\theta_k)\le (F_j^{(t)})^{-1}(\delta)\right\} = F_j^{(t)}(F_j^{(t)})^{-1}(\delta) = \delta.\]
	\indent Thus, $\Delta_j^{(t)}$ has a uniform distribution on $[0,1]$ for all $t$.\\
	{\bf Step 2}: From uniform distribution on $[0,1]$ to $F_j^{\ast}(z)$\\
	\indent Define
	\[\eta_j^{(t)}\triangleq (F_j^{\ast})^{-1}(\Delta_j^{(t)}).\]
	\indent Since $\Delta_j^{(t)}$ has a uniform distribution, we get 
	\[P\{\eta_j^{(t)}\le z\} = P\{\Delta_j^{(t)}\le F_j^{\ast}(z)\} = F_j^{\ast}(z).\]
	\indent Note that conditioned on the action profiles, the sequence $\{z_j^{(t)}(\theta^{\ast},\theta_k)\}$ is independent over $t$. Therefore, the sequence $\{\eta_j^{(t)}\}_{t\ge 1}$ is independent and identically distributed over $t$ and has the cumulative distribution function $F_j^{\ast}(z)$. Furthermore, by the strong law of large numbers,
	\begin{equation}\label{strong_law}
		\lim_{T\to\infty}\frac{1}{T}\sum_{t=1}^T\eta_j^{(t)} = \mathbb{E}\left[\log\frac{f_j(y_j|x^{\ast}(\mu),\theta^{\ast})}{f_j(y_j|x^{\ast}(\mu),\theta_k)}\right]\quad a.s.
	\end{equation}
	
	On the other hand, since $x^{(t)}\to x^{\ast}(\mu)\ a.s.$, we get by the continuity of the likelihood function (Assumption 1) that $z_j^{(t)}(\theta^{\ast},\theta_k)\to \log\frac{f_j(y_j|x^{\ast}(\mu),\theta^{\ast})}{f_j(y_j|x^{\ast}(\mu),\theta_k)}$, a.s. Thus, the convergence also holds in distribution, i.e., for all $z\in\mathbb{R}$,
	\[\lim_{t\to\infty}F_j^{(t)}\left(z\right) = \lim_{t\to\infty}F_j^{\ast}\left(z\right).\]
	Moreover, since the set of discontinuous points of a monotone function is at most countable \cite[Theorem 4.30]{rudin1976principles}, we get
	\begin{equation}\label{strong_law_error}
		\lim_{t\to\infty}\left(z_j^{(t)}(\theta^{\ast},\theta_k) - \eta_j^{(t)}\right) = \lim_{t\to\infty}\left(z_j^{(t)}(\theta^{\ast},\theta_k) - (F_j^{\ast})^{-1}F_j^{(t)}\left(z_j^{(t)}(\theta^{\ast},\theta_k)\right)\right) = 0\quad a.s.
	\end{equation} 
	By \eqref{strong_law}-\eqref{strong_law_error} and using Toeplitz's lemma (Lemma 4) with $A_{Tt} = \frac{1}{T}$ ($t = 1,\dots,T$) and $Y_t = z_j^{(t)}(\theta^{\ast},\theta_k) - \eta_j^{(t)}$, we obtain
	\begin{equation}
		\lim_{T\to\infty}\frac{1}{T}\sum_{t=1}^Tz_j^{(t)}(\theta^{\ast},\theta_k) = \mathbb{E}\left[\log\frac{f_j(y_j|x^{\ast}(\mu),\theta^{\ast})}{f_j(y_j|x^{\ast}(\mu),\theta_k)}\right]\quad a.s.
	\end{equation}
	
	Next, we associate the weighted sum \eqref{weighted_sum} with $\lim_{T\to\infty}\frac{1}{T}\sum_{t=1}^Tz_j^{(t)}(\theta^{\ast},\theta_k)$. From \eqref{weighted_sum}, we further derive 
	\begin{align*}
		&\quad\ \lim_{T\to\infty}\frac{1}{\sum_{t=1}^T\alpha^{(t)}}\log\frac{\mu_i^{(T+1)}(\theta^{\ast})}{\mu_i^{(T+1)}(\theta_k)}\\
		&= \frac{1}{N}\sum_{j=1}^N\lim_{T\to\infty}\frac{1}{s^{(T)}}\left(T\alpha^{(T)}\cdot\frac{1}{T}\sum_{t=1}^Tz_j^{(t)}(\theta^{\ast},\theta_k) + \sum_{t=1}^{T-1}t(\alpha^{(t)} - \alpha^{{(t+1)}})\cdot\frac{1}{t}\sum_{\tau=1}^tz_j^{(\tau)}(\theta^{\ast},\theta_k)\right).
	\end{align*}
	Notice that $T\alpha^{(T)} + \sum_{t=1}^{T-1}t(\alpha^{(t)} - \alpha^{(t+1)}) = s^{(T)}$. As a result, by $s^{(T)}\to\infty$ ($T\to\infty$), applying Toeplitz's lemma again with $A_{Tt} = \frac{t(\alpha^{(t)} - \alpha^{(t+1)})}{s^{(T)}}$ ($t = 1,\dots,T-1$), $A_{TT} = \frac{T\alpha^{(T)}}{s^{(T)}}$, and $Y_t = \frac{1}{t}\sum_{\tau=1}^tz_j^{(\tau)}(\theta^{\ast},\theta_k)$ yields
	\begin{align}
		\lim_{T\to\infty}\frac{1}{\sum_{t=1}^T\alpha^{(t)}}\log\frac{\mu_i^{(T+1)}(\theta^{\ast})}{\mu_i^{(T+1)}(\theta_k)}
		&= \frac{1}{N}\sum_{j=1}^N\lim_{T\to\infty}\frac{1}{T}\sum_{t=1}^Tz_j^{(t)}(\theta^{\ast},\theta_k)\notag\\
		&= \mathbb{E}\left[\frac{1}{N}\sum_{j=1}^N\log\frac{f_j(y_j|x^{\ast}(\mu),\theta^{\ast})}{f_j(y_j|x^{\ast}(\mu),\theta_k)}\right]\quad a.s. \label{Toeplitz}
	\end{align}
	
	Recall that $Z(\theta^{\ast},\theta_k)\triangleq \frac{1}{N}\sum_{j=1}^ND_{KL}(f_j(y_j|x^{\ast}(\mu),\theta^{\ast})\|f_j(y_j|x^{\ast}(\mu),\theta_k))$. From the definition \eqref{KL-div_def} of KL divergence,
	\[Z(\theta^{\ast},\theta_k) = \mathbb{E}\left[\frac{1}{N}\sum_{j=1}^N\log\frac{f_j(y_j|x^{\ast}(\mu),\theta^{\ast})}{f_j(y_j|x^{\ast}(\mu),\theta_k)}\right].\]
	By Assumption \ref{asm2}, $Z(\theta^{\ast},\theta_k)$ is strictly positive. Additionally, from \eqref{Toeplitz}, for all $\epsilon > 0$, there exists $T'$ such that for all $T\ge T'$,
	\[\left|\frac{1}{s^{(T)}}\log\frac{\mu_i^{(T+1)}(\theta^{\ast})}{\mu_i^{(T+1)}(\theta_k)} - Z(\theta^{\ast},\theta_k)\right|\le\epsilon\quad a.s.\]
	Therefore,
	\[\frac{\mu_i^{(T+1)}(\theta_k)}{\mu_i^{(T+1)}(\theta^{\ast})} \le \exp(-s^{(T)}(Z(\theta^{\ast},\theta_k) - \epsilon))\quad a.s.\]
	Using $\sum_{k=1}^M\mu_i^{(T+1)}(\theta_k) = 1$, we get
	\[\frac{1}{\mu_i^{(T+1)}(\theta^{\ast})} - 1\le \sum_{\theta_k\neq\theta^{\ast}}\exp(-s^{(T)}(Z(\theta^{\ast},\theta_k) - \epsilon))\quad a.s.\]
	Furthermore, we derive
	\begin{equation}\label{convergence_rate}
		\frac{1}{1 + \sum_{\theta_k\neq\theta^{\ast}}\exp(-s^{(T)}(Z(\theta^{\ast},\theta_k) - \epsilon))}\le \mu_i^{(T+1)}(\theta^{\ast})\le 1\quad a.s.
	\end{equation}
	Hence we have that $\mu_i^{(t)}(\theta^{\ast})\to 1\ a.s.$, and the assertion is proved.
\end{proof}
Theorem 3 establishes the convergence of the beliefs to the true parameter and the rate is given by \eqref{convergence_rate}. Combining Theorem 3 and Theorem 2, we further obtain that the strategy profile converges to the true Nash equilibrium. 
\begin{remark}
	The equation \eqref{convergence_rate} shows that the convergence rate depends on the step-size sequence $\{\alpha^{(t)}\}$. Combined with Theorem 1 in \cite{eksin2017distributed}, we know that if we choose an appropriate sequence $\{\alpha^{(t)}\}$ such that $\mu_i^{(t)}$ converges to $\mu$ at a rate faster than $O(\frac{\log t}{t})$, agents may replace the best response strategies with a distributed fictitious play scheme. 
\end{remark}
\section{Conclusion}
In this paper, we studied a Nash equilibrium seeking problem in games with a cost-relevant parameter. In order to learn the true parameter while finding the Nash equilibrium, we introduced a diminishing step-size sequence on the basis of the traditional non-Bayesian learning rule to ensure that the beliefs of agents on the parameter reach consensus. Then we combined the best response dynamics to update the strategies, and thus, proposed a novel distributed non-Bayesian rule. Using the strong law of large numbers and Toeplitz's lemma, we showed that agents' beliefs concentrate on the true parameter and the strategy profile converges to the true Nash equilibrium.  
\bibliographystyle{IEEEtran}
\bibliography{non_Bayesian_learning.bib}
\end{document}